\documentclass[10pt]{amsart}
\usepackage{amssymb,amsbsy,amsmath,amsfonts,amssymb}
\usepackage{latexsym,euscript,exscale}

\title{Determinacy of adversarial Gowers games}
\author {Christian Rosendal}
\address{Department of Mathematics, Statistics, and Computer Science (M/C 249)\\
University of Illinois at Chicago\\
851 S. Morgan St.\\
Chicago, IL 60607-7045\\
USA}
\email{rosendal.math@gmail.com}
\urladdr{http://homepages.math.uic.edu/$_~$rosendal}

\thanks{The initial research for this article was done while the author was visiting V. Ferenczi at the
University of S\~ao Paulo, Brazil, with the support of FAPESP. This work was partially supported by a grant from the Simons Foundation (\#229959 to Christian Rosendal).
The author's research was likewise
supported by NSF grants  DMS 0919700, DMS 0901405 and DMS 1201295.}
\newcommand{\compl}{{\sim\!}}

\newcommand {\A}{\mathbb A}
\newcommand {\F}{\mathbb F}
\newcommand {\N}{\mathbb N}

\newcommand{\G}{\mathbb G}

\newcommand {\OO}{\mathbb O}

\newcommand{\om}{\omega}

\newcommand{\con}{\;\hat{}\;}

\newcommand{\tom} {\emptyset}

\newcommand{\saa}{\Rightarrow}
\newcommand{\equi}{\Longleftrightarrow}

\newcommand {\del}{ \; \big| \;}

\newcommand {\go} {\mathfrak}

\newcommand {\e} {\exists}
\renewcommand {\a} {\forall}

\newtheorem{thm}{Theorem}[section]

\newtheorem{lemme}[thm]{Lemma}
\newtheorem{prop} [thm] {Proposition}
\newtheorem{defi} [thm] {Definition}

\newtheorem{claim}[thm] {Claim}

\newtheorem{nota}[thm]{Notation}

\theoremstyle{definition}

\newtheorem{prob}[thm]{Problem}
\newtheorem{obs}[thm] {Observation}

\begin{document}



\maketitle
\begin{abstract}
We prove a game theoretic dichotomy for $G_{\delta\sigma}$ sets of block sequences in vector spaces that extends, on the one hand, the block Ramsey theorem of W. T. Gowers proved for analytic sets of block sequences and, on the other hand, M. Davis' proof of ${\bf \Sigma}^0_3$ determinacy.
\end{abstract}

\section{Introduction}
In the present paper, we  prove an extension of W. T. Gowers' Ramsey theorem for block sequences in normed vector spaces \cite{gowers}. This was instrumental in the proof of his dichotomy for Banach spaces between containing an unconditional basic sequence or a hereditarily indecomposable subspace that ultimately led to a solution of the homogeneous space problem for Banach spaces.

The statement of Gowers' theorem is as follows. Assume that $\A$ is an analytic set of sequences $(y_n)$  of normalised vectors in a separable Banach space $E$ and, moreover,  any infinite-dimensional subspace $X\subseteq E$ contains a sequence from $\A$. Then there is an infinite-dimensional subspace $X\subseteq E$ for which one can sequentially choose the terms of some $(y_n)$ close to $\A$ such that the $y_n$ belong to any given infinite-dimensional subspaces $Y_n\subseteq X$. 

The precise statement is formulated in terms of a game in which player I plays the subspaces $Y_n\subseteq X$, while player II choses the vectors $y_n\in Y_n$. While we shall not follow Gowers' lead in dealing with normed vector spaces, but instead use the set-up of \cite{exact} and thus consider only vector spaces over countable fields, the exact results proved here easily imply slightly stronger, but approximate, statements for normed vector spaces as is shown in \cite{exact}. 

So suppose that $E$ is a countable-dimensional vector space over a countable field $\mathfrak F$. We define the {\em Gowers game} $G_X$ played below an infinite-dimensional subspace $X\subseteq E$ as follows. Players I and II alternate in playing respectively infinite-dimensional subspaces $Y_n\subseteq X$ and non-zero vectors $y_n\in Y_n$,
$$
{
\begin{array}{cccccccccccc}
{\bf I} & &Y_0 &&Y_1  & &Y_2  &&\ldots \\
{\bf II} & & &y_0\in Y_0  & &  y_1\in Y_1&    &y_2\in Y_2&\ldots
\end{array}
}
$$
Similarly, the {\em infinite asymptotic game} $F_X$ is defined as the Gowers game except that I is now required to play subspaces $Y_n$ of finite codimension in $X$ (in fact, even so-called {\em tail subspaces}). Thus, from the viewpoint of II, the game has not changed, but, in $F_X$, player  I will have significantly less control over where player II chooses his vectors. In both games, we say that the infinite sequence $(y_n)$ produced is the {\em outcome} of the game. 

Note that $E$ is a countable set and therefore the infinite power $E^\infty$ is a Polish space, i.e., separable and completely metrisable, when $E$ is endowed with the discrete topology. The version of Gowers' theorem proved in \cite{exact} states that if $\A\subseteq E^\infty$ is an analytic set, i.e., a continuous image of a Polish space, then there is an infinite-dimensional subspace $X\subseteq E$ such that either player I has a strategy in $F_X$ to force the outcome to lie in $\compl\, \A$ or player II has a strategy in $G_X$ to play into $\A$. We remark that, on the one hand, this is stronger than simply stating that the game $G_X$ is determined, since, clearly, if I has a strategy in $F_X$ to force the outcome to lie in $\compl\, \A$, then he also has a strategy in $G_X$ to the same effect. On the other hand, this strong determination comes at the price of passing to the subspace $X\subseteq E$, which indicates that Ramsey theory is involved. 

In \cite{pelczar}, A. M. Pelczar studied a variant of the Gowers game in which both players are directly contributing to the outcome. This was further refined in \cite{minimal} and, in  \cite{exact}, was formulated as the determinacy of two related {\em adversarial Gowers games}, $A_X$ and $B_X$. 

Again, for $X\subseteq E$ an infinite-dimensional subspace, we define the game $A_X$ by combining the games $G_X$ and $F_X$, letting player I of $A_X$ act simultaneously as player II of $F_X$ and player I of $G_X$, while player II of $A_X$ acts as player I of $F_X$ and player II of $G_X$. Concretely, I plays subspaces $Y_n\subseteq X$ and non-zero vectors $x_n$, while II plays subspaces $X_n\subseteq X$ and non-zero vectors $y_n$, satisfying $x_n\in X_n$ and $y_n\in Y_n$, 
$$
{
\begin{array}{cccccccccccc}
{\bf I} & & &Y_0, x_0\in X_0 &  & Y_1, x_1\in X_1 &  &\ldots \\
{\bf II} & &X_0 &  & X_1, y_0\in Y_0 &  & X_2, y_1\in Y_1   &\ldots
\end{array}
}
$$
Moreover,  $X_0,X_1,\ldots$ are required to have finite codimension in $X$ (again, they can be taken to be tail subspaces with respect to a given basis for $E$) and $Y_0,Y_1,\ldots$ are arbitrary infinite-dimensional subspaces of $X$.

The game $B_X$ is defined as the game $A_X$ except that we now require the spaces $Y_n$ to have finite codimension in $X$, while instead the $X_n$ can be arbitrary infinite-dimensional subspaces. Thus, in the game $A_X$, it is player I that have relatively tight control over the sequence of vectors played by II, since I is the one to play arbitrary infinite-dimensional subspaces of $X$. On the contrary, in the game $B_X$, the roles of I and II are reversed. In both games, the {\em outcome} is defined to be the infinite sequence $(x_0,y_0,x_1,y_1, \ldots)\in E^\infty$. 

Extending Theorem 12 in \cite{exact} for closed $\A$, the main result of our paper is the following. 
\begin{thm}\label{Gdeltasigma}
Suppose $\A\subseteq E^\infty$ is  $F_{\sigma\delta}$ or $G_{\delta\sigma}$. Then there is an infinite-dimensional subspace $X\subseteq E$
such that either
\begin{enumerate}
  \item II has a strategy in $A_X$ to play in $\compl\, \A$, or
  \item I has a strategy in $B_X$ to play in $\A$.
\end{enumerate}
\end{thm}
As above, Theorem \ref{Gdeltasigma} is significantly stronger than merely requiring the games $A_X$ and $B_X$ to play into $\A$ to be determined, which of course is well-known. The main interest of the result lies in the fact that it provides a strong form of determinacy, namely, in each case, the winning player has a strategy in a game that a priori is particularly hard for him to play. However, again this comes at the cost of passing to the subspace $X\subseteq E$. 

It remains an open problem how much  this theorem can be improved. In the light of D. A. Martin's proof of Borel determinacy \cite{martin1,martin2}, it is tempting to believe that it should hold for Borel sets $\A$. However, this extension seems less than straightforward. The problem lies in combining Martin's proof (or the final result) with the Ramsey theoretical techniques necessary for the proof of Theorem \ref{Gdeltasigma}. The original proofs of determinacy for the first three levels of the Borel hierarchy  (D. Gale and F. M. Stewart \cite{gale} for open, P. Wolfe \cite{wolfe} for $G_\delta$ and M. Davis \cite{davis} for $G_{\delta\sigma}$) all proceed in second order arithmetic and thus the types of objects quantified over are at most subsets of the integers. This implies that these proofs commute sufficiently with our relatively simple Ramsey theory to be able to go through in our setting, albeit with some additional complications. 

On the other hand, the proofs of determinacy for more complicated Borel sets, J. Paris \cite{paris} for ${\bf \Sigma}^0_4$ and Martin for all of ${\bf \Delta}^1_1$, necessarily requires  a larger part of the set theoretical universe as shown by H. Friedman \cite{friedman} and recent refinements by A. Montalban and R. Shore \cite{montalban}. Thus, the proof of Borel determinacy demands $\om_1$ iterations of the power set operation and thus the existence of much larger sets than those involved in the actual statement of determinacy. The same of course also applies to proofs of determinacy in the presence of large cardinals, such as Martin's proof of analytic determinacy from a measurable cardinal \cite{martin3}. As a consequence, to prove determinacy of Borel games with moves in our space $E$, one  is led to consider other games on much larger sets for which the Ramsey theory loses meaning. Alternatively, the existence of large cardinals may itself lead to a better tree representation of Borel or analytic sets, which could prove useful for lifting Theorem \ref{Gdeltasigma} to general Borel sets.

It is of course quite possible that our theorem does not generalise to higher order Borel sets and thus the difference in proof theoretic strength between  ${\bf \Sigma}^0_3$ and Borel determinacy translates into a difference in truth value for the adversarial Gowers games.

\begin{prob}
Does Theorem \ref{Gdeltasigma} hold for all Borel sets $\A$ or even for analytic sets in the presence of large cardinals?
\end{prob}

As a last few words on these issues, let us also mention that, in a precise sense, the class of {\em adversarially Ramsey sets}, i.e., the class of sets $\A\subseteq  E^\infty$ satisfying the conclusion of Theorem \ref{Gdeltasigma}, is smaller than that of determined sets. That is, if $\bf \Gamma$ is a class of subsets of Polish spaces closed under continuous preimages and such that any $\bf \Gamma$ subset of $E^\infty$ is adversarially Ramsey, then any game on $\N$ to play in a $\bf \Gamma$ subset of $\N^\N$ is determined. To see this, we note that we can code elements of $\N^\N$ using sequences $(x_n)\in E^\infty$ by letting each $x_k\in E$ code a natural number by the coefficient of its first non-zero coordinate with respect to a fixed basis $(e_n)$ for $E$. So the limits of determinacy on games on $\N$ also limits the class of adversarially Ramsey sets. 

Finally, let us dispel a possible source of confusion concerning issues of determinacy. Though, formally, players I and II in the adversarial Gowers games play object of higher type, namely the infinite-dimensional subspaces $X_n$ and $Y_n$, this can easily be circumvented. For example, instead of letting I play all of the subspace $X_k$ at once, we can simply let him successively play the vectors of a basis for $X_k$ and allow II to wait to play a vector until he can find one that is a linear combination of the part of the basis that I has played thus far. As shown  by B. Velickovic in \cite{jordi}, this produces an equivalent game in which all moves now are of lower type, namely vectors in $E$. 

For applications of the above dichotomies to the geometry of Banach spaces, we refer the reader to \cite{gowers}, \cite{minimal} and \cite{alpha}.

\noindent{\em Acknowledgement}: The author is grateful for a number of insightful comments and useful discussions with A. Montalban, J. Moore and P. Welch on the topic of this paper.

\section{Notation}
Fix  a countable field $\mathfrak F$ and let $E$ be the countable-dimensional $\go
F$-vector space with basis $(e_n)$. We shall use $x,y,z,v$ as variables for {\em non-zero}
elements of $E$. If $x=\sum a_ne_n\in E$, the {\em support} of $x$ is the finite non-empty set  
$$
{\rm supp}\; x=\{n\del a_n\neq 0\}.
$$
A finite or infinite sequence $(x_0,x_1,x_2,x_3,\ldots)$ of non-zero vectors is said to be a {\em block sequence}
if  
$$
\max {\rm supp}\; x_n<\min{\rm supp}\; x_{n+1}
$$ 
for all $n$. In particular, the terms of a block sequence are linearly independent.

Notice that, by elementary linear algebra, for all infinite dimensional
subspaces $X\subseteq E$ there is a subspace $Y\subseteq X$ spanned by an
infinite block sequence, called a {\em block subspace}. So, henceforth, we use
variables $X,Y,Z, V, W$ to denote infinite dimensional block subspaces of $E$. Also, if $X\subseteq E$ is a block subspace and $k$ a natural number, we let $X[k]=\{x\in X\del k<\min {\rm supp}\; x\}$, which is a cofinite-dimensional so-called {\em tail subspace} of $X$.
Finally, we denote infinite block sequences by variables $\bf x,y,z$ and finite block
sequences by variables $\vec x, \vec y, \vec z$. We use the symbols $\sqsubseteq$ and $\sqsubset$ to denote end-extension, respectively, proper end-extension, of finite sequences.

If $X$ and $Y$ are block subspaces, we write $Y\subseteq^*X$ to denote that $Y[n]\subseteq X$ for some sufficiently large $n$ and so, in particular, that $Y\cap X$ has finite codimension in $Y$.
A principle, that will be used repeatedly here, is the fact that if $X_0\supseteq X_1\supseteq \ldots$ is an infinite descending sequences of block subspaces then there is a block subspace $Y\subseteq X_0$ such that $Y\subseteq^*X_n$ for all $n$. To see this, suppose that each $X_n$ is spanned by a block sequence $(x_k^n)_k$ and note that then $(x_0^0,x_1^1,x_2^2,\ldots)$ is also a block sequence. Moreover, if $Y=[x_0^0,x_1^1,x_2^2,\ldots]$ denotes the linear span, then $Y\subseteq^*X_n$ for all $n$.

We equip $E$ with the discrete topology,
whereby any subset is open, and equip its countable power $E^\infty$ with the
product topology. Since $E$ is a countable discrete set, $E^\infty$ is a Polish
space. Notice that a basis for the topology on $E^\infty$ is given by sets of
the form
$$
N_{(x_0,\ldots,x_k)}=\{(y_n)\in E^\infty\del y_0=x_0\;\&\;\ldots\;\&\;y_k=x_k\},
$$
where $x_0,\ldots,x_k\in E$ (possibly zero vectors). Finally, $E^{<\infty}$ will denote the set of finite block sequences in $E$.


\section{Adversarial games}
\subsection*{The game $A_V(\vec v)$}
Suppose $V\subseteq E$. We define the game $A_V$ played below $V$ between two
players I and II as follows: I and II alternate  in choosing block subspaces
$Z_0,Z_1,Z_2,\ldots\subseteq V$ and vectors $x_0,x_1,x_2,\ldots\in V$,
respectively natural numbers $n_0,n_1,n_2,\ldots$ and vectors $y_0,y_1,y_2,\ldots\in
V$ according to the constraints $ x_i\in V[n_i]$ and $y_i\in Z_i$:
$$
{
\begin{array}{cccccccccccc}
{\bf I} & & &x_0\in V[n_0], Z_0 &  & x_1\in V[n_1],Z_1 &  &\ldots \\
{\bf II} & &n_0 &  & y_0\in Z_0,n_1 &  & y_1\in Z_1,n_2   &\ldots
\end{array}
}
$$
We say that the sequence $(x_0,y_0,x_1,y_1,\ldots)$ is the {\em outcome} of the
game.

If $\vec v$ is a finite block sequence of {\em even} length, the game $A_V(\vec
v)$ is defined as above except that the outcome is now $\vec v\con
(x_0,y_0,x_1,y_1,\ldots)$.

On the other hand, if $\vec v$ is a finite block sequence of {\em odd} length,
$A_V(\vec v)$ is defined in a similar way as before except that I begins the game:
$$
\begin{array}{cccccccccccc}
{\bf I} & & Z_0 &  & x_0\in V[n_0],Z_1 &  & x_1\in V[n_1],Z_2 &\ldots \\
{\bf II} &  &  & y_0\in Z_0,n_0 &  & y_1\in Z_1,n_1 &  &\ldots
\end{array}
$$
and the {\em outcome} is now $\vec v\con(y_0,x_0,y_1,x_1,\ldots)$ rather than
$\vec v\con(x_0,y_0,x_1,y_1,\ldots)$.

\subsection*{The game $B_V(\vec v)$}
We define the game $B_V$ in a similar way to $A_V$ except that we now have I
playing integers and II playing block subspaces:
$$
\begin{array}{cccccccccccc}
{\bf I} &  &  & x_0\in Z_0,n_0 &  & x_1\in Z_1,n_1 &  &\ldots       \\
{\bf II} & & Z_0 &  & y_0\in V[n_0],Z_1 &  & y_1\in V[n_1],Z_2 &\ldots
\end{array}
$$
with $x_i\in Z_i\subseteq V$ and $y_i\in V[n_i]$. Again, the {\em outcome} is
$(x_0,y_0,x_1,y_1,\ldots)$.

If $\vec v$ is a finite block sequence of {\em even} length, the game $B_V(\vec
v)$ is defined as above except that the outcome is now $\vec v\con
(x_0,y_0,x_1,y_1,\ldots)$.

On the other hand, if $\vec v$ is a finite block sequence of {\em odd} length,
$B_V(\vec v)$ is defined by letting I begin:
$$
\begin{array}{cccccccccccc}
{\bf I} & &n_0 &  & x_0\in Z_0,n_1 &  & x_1\in Z_1,n_2 &  \ldots    \\
{\bf II} & & &y_0\in V[n_0], Z_0 &  &y_1\in V[n_1],Z_1 &  &\ldots
\end{array}
$$
and the {\em outcome} is now $\vec v\con (y_0,x_0,y_1,x_1,\ldots)$.

\

Thus, in both games $A_V$ and $B_V$, one should remember that I is the {\em
first} to play a vector. And in $A_V$, I plays block subspaces and II plays
tail subspaces, while in $B_V$, II takes the role of playing block subspaces
and I plays tail subspaces.

The central distinction between the two games lies in the fact that it is, in general, much easier to pick vectors in a tail subspace than in an arbitrary block subspace. Thus, in the game $A_V$, player II has to make choices of vectors in potentially coinfinite-dimensional subspaces $Z_i\subseteq V$ selected by I, while only being able to force I to make his choices of vectors in subspaces of finite codimension. 
So $A_V$ is harder to play for II than the game $B_V$, while the opposite is the case for player I.


\section{Quasistrategies}
A strategy  for II in the game $A_V(\vec v)$ is a function that to each position of the game in which II is to play, say $(n_0,x_0,Z_0,y_0,n_1,\ldots, y_k,n_{k+1}, x_{k+1}, Z_{k+1})$, associates the next required move of II. 
Alternatively, the strategy can be identified with the set of positions that have been played according to the strategy.
However, since the outcome only consists of the sequence of vectors $x_0,y_0,x_1,y_1,\ldots$, it is reasonable to expect that strategies should also only depend on the initial part of the outcome, i.e., $(x_0,y_0,\ldots,x_k,y_k, x_{k+1})$, together with the condition $Z_{k+1}$, rather than on $(n_0,x_0,Z_0,y_0,n_1,\ldots, y_k,n_{k+1}, x_{k+1}, Z_{k+1})$. Now, as we shall verify, this means that a quasistrategy for II in $A_V(\vec v)$ can be seen as a subset of $E^{<\infty}$ with certain extension properties. 
Moreover, as these quasistrategies for II in $A_V(\vec v)$ will be used as self imposed rules for II in $B_V(\vec v)$, we shall denote them as rules rather than quasistrategies.

\begin{defi}
Assume $V\subseteq E$ and $\vec v$ is a finite block sequence. A {\em $(V,\vec v)$-rule}  is a subset  $T\subseteq E^{<\infty}$ such that $\vec v \in T$ and 
\begin{enumerate}
\item[(i)] if $\vec y\in T$ and $|\vec y|$ is odd, then for any $Z\subseteq V$ there is some $z\in Z$ such that $\vec y\con z\in T$,
\item[(ii)] if $\vec y\in T$ and $|\vec y|$ is even, then there is some $n$ such that $\vec y\con z\in T$ for all $z\in V[n]$.
\end{enumerate}
\end{defi}

\begin{nota}
Let $T\subseteq E^{<\N}$ be any subset. We define the set of {\em infinite branches} of $T$ by
$$
[T]=\{(x_n)\in E^\infty\del \e^\infty m\;(x_0,x_1,\ldots,x_m)\in T\}.
$$
Also, if $\vec x$ is a finite block sequence, we let 
$$
T_{\vec x}=\{\vec y\del \vec x\con \vec y\in T\}.
$$
\end{nota}

The following result is not used in the proof of the main result, but clarifies the nature of quasistrategies. 

\begin{prop}\label{quasistrategy}
Let $\A\subseteq E^\infty$, $V\subseteq E$ and $\vec v$ be a finite block sequence. Then II has a strategy $\sigma$ in $A_V(\vec v)$ to play in $\A$ if and only if there is a $(V,\vec v)$-rule $T$  such that $[T]\subseteq \A$.
\end{prop}

\begin{proof}Obviously, any $(V,\vec v)$-rule $T$ such that $[T]\subseteq \A$ provides a strategy $\sigma$ for II in $A_V(\vec v)$ to play in $\A$. For it suffices that $\sigma$ ensures that every position of the game belongs to $T_{\vec v}$.

Conversely, suppose that $\sigma$ is a strategy for II to $A_V(\vec v)$ to play in $\A$ and assume that $\vec v$ has even length,  the case when the length is odd being similar. We identify the strategy $\sigma$ with the tree of legal positions in $A_V(\vec v)$ in which II has played according to the strategy.
We define $T_{\vec v}\cap E^n$ by induction on $n$ simultaneously with a monotone function $\phi$ assigning to each element of $T_{\vec v}$ some position in $A_V(\vec v)$ in which II has played according to $\sigma$. More precisely, $\phi$ takes values of the following form
\begin{displaymath}\begin{split}
\phi(x_0,y_0,&x_1,y_1,\ldots,x_k,y_k)\\
&=(n_0,x_0,Z_0,y_0,n_1,x_1,Z_1,y_1,\ldots,n_k,x_k,Z_k,y_k,n_{k+1})\in \sigma
\end{split}\end{displaymath}
and
$$
\phi(x_0,y_0,x_1,y_1,\ldots,x_k)=(n_0,x_0,Z_0,y_0,n_1,x_1,Z_1,y_1,\ldots,n_k,x_k).
$$

\noindent (i) First, let $\vec v\in T$ and let $\phi(\tom)=n_0$, where $n_0$ is the first play of II according to the strategy $\sigma$. 

\noindent (ii) Now, if $(x_0,y_0,x_1,y_1,\ldots,x_k,y_k)\in T_{\vec v}$ and 
\begin{displaymath}\begin{split}
\phi(x_0,y_0,&x_1,y_1,\ldots,x_k,y_k)\\
&=(n_0,x_0,Z_0,y_0,n_1,x_1,Z_1,y_1,\ldots,n_k,x_k,Z_k,y_k,n_{k+1})\in \sigma
\end{split}\end{displaymath}
has been defined, we let $(x_0,y_0,x_1,y_1,\ldots,x_k,y_k,x_{k+1})\in T_{\vec v}$ for all $x_{k+1}\in X[n_{k+1}]$ and set
\begin{displaymath}\begin{split}
\phi(x_0,y_0,&x_1,y_1,\ldots,x_k,y_k,x_{k+1})\\
&=(n_0,x_0,Z_0,y_0,n_1,x_1,Z_1,y_1,\ldots,n_k,x_k,Z_k,y_k,n_{k+1},x_{k+1}).
\end{split}\end{displaymath}

\noindent (iii) And, if $(x_0,y_0,x_1,y_1,\ldots,x_k)\in T_{\vec v}$ and 
\begin{displaymath}\begin{split}
\phi(x_0,y_0,&x_1,y_1,\ldots,x_k)=(n_0,x_0,Z_0,y_0,n_1,x_1,Z_1,y_1,\ldots,n_k,x_k)
\end{split}\end{displaymath}
has been defined, we put $(x_0,y_0,x_1,y_1,\ldots,x_k,y_k)\in T_{\vec v}$ if there are $Z_k\subseteq X$ and $n_{k+1}$ such that $y_k\in Z_k$ and 
$$
(n_0,x_0,Z_0,y_0,n_1,x_1,Z_1,y_1,\ldots,n_k,x_k,Z_k,y_k,n_{k+1})\in\sigma.
$$
In this case, we choose any such $Z_k$ and $n_{k+1}$ and let 
\begin{displaymath}\begin{split}
\phi(x_0,y_0,&x_1,y_1,\ldots,x_k,y_k)\\
&=(n_0,x_0,Z_0,y_0,n_1,x_1,Z_1,y_1,\ldots,n_k,x_k,Z_k,y_k,n_{k+1}).
\end{split}\end{displaymath}

Since $\sigma$ is a strategy for II, it is easy to verify that $T$ defined as above is a $(V,\vec v)$-rule (in fact, $T_{\vec v}$ is also a pruned tree). Also, if $(x_0,y_0,x_1,y_1,x_2,y_2,\ldots)\in T_{\vec v}$, we see that
$$
\phi(\tom)\sqsubseteq \phi(x_0)\sqsubseteq\phi(x_0,y_0)\sqsubseteq \phi(x_0,y_0,x_1)\sqsubseteq\ldots
$$
and $\phi(x_0,y_0,\ldots,x_k,y_k)\in \sigma$ for all $k$. So $\bigcup_k\phi(x_0,y_0,\ldots,x_k,y_k)\in [\sigma]$. Since the strategy $\sigma$ plays in $\A$, it follows that the outcome $\vec v\con (x_0,y_0,x_1,y_1,x_2,y_2,\ldots)$ of the run  $\bigcup_k\phi(x_0,y_0,\ldots,x_k,y_k)$ of the game $A_V(\vec v)$ must lie in $\A$. Thus $[T]\subseteq \A$.
\end{proof}

As mentioned above, a $(V,\vec v)$-rule $T$ can be viewed as a self imposed rule for player II in the game $B_V(\vec v)$.

\begin{defi}
Suppose  a $(V,\vec v)$-rule  $T$ is given. The {\em $T$-induced subgame} $B_V^T(\vec v)$ of $B_V(\vec v)$ is played as the game $B_V(\vec v)$ subject to the additional condition that any position $(x_0,y_0,\ldots,x_k,y_k)$ or $(x_0,y_0,\ldots,x_k)$ 
in $B_V^T(\vec v)$ should belong to $T_{\vec v}$.
 \end{defi}
 
 We note that the game $B^T_V(\vec v)$ imposes essentially no new requirements on player I compared to the game $B_V(\vec v)$. This is because if $\vec x\in T_{\vec v}$ is a position of the game so that I is to play, i.e., $|\vec v\con \vec x|$ is even, then there is an $m$ such that $\vec v\con \vec x\con y\in T$ for all $y\in V[m]$. So player I may just assume that he is responding to some $n\geqslant m$ played by II in $B_V(\vec v)$.

\begin{obs}We remark that if $T$ is a $(V,\vec v)$-rule, then for any $W\subseteq^*V$ and any $\vec w\in T$,  $T$ is an $(W,\vec w)$-rule.\end{obs}


\section{Proof of the main theorem}
Note that, in the game $B^T_V(\vec v)$, player I has to pick vectors in arbitrarily small subspaces chosen by II while only being able to force II to pick vectors in tail subspaces. Therefore, by the asymptotic nature of the games, it is easy to see that any strategy for I in $B^T_V({\vec v})$ to play into a set $\A$ immediately provides a strategy for I in $B^T_W(\vec v)$ to play in $\A$ as long as $W\subseteq ^*V$.

\begin{obs}\label{hered}
Suppose $\A\subseteq E^\infty$, $W\subseteq^*V$, $\vec v\in E^{<\infty}$ and a $(V,\vec v)$-rule $T\subseteq E^{<\infty}$ are given such that
$$
\text{I has a strategy in $B^T_V(\vec v)$ to play in $\A$}.
$$
Then also
$$
\text{I has a strategy in $B^T_W(\vec v)$ to play in $\A$}.
$$
\end{obs}
On the other hand, while I may not have a strategy in $B_V(\vec v)$ to play into $\A$, he could have one in $B_W(\vec v)$ for some $W\subseteq^*V$. This situation is remedied by the following simple diagonalisation.

\begin{lemme}\label{diag}
Suppose $\A\subseteq E^\infty$, $V\subseteq E$, $\vec v\in E^{<\infty}$ and a $(V,\vec v)$-rule $T\subseteq E^{<\infty}$ are given. Then there is $X\subseteq V$ so that, for all $\vec x\in T$,
\[\begin{split}
\e Y\subseteq X, &\text{ I has a strategy in $B^T_Y(\vec x)$ to play in $\A$ }\\
&
\equi\;\;
\a Y\subseteq X, \text{ I has a strategy in $B^T_Y(\vec x)$ to play in $\A$}.
\end{split}\]
\end{lemme}

\begin{proof}Enumerate $T$ as $\vec x_0,\vec x_1,\ldots$ and define a sequence of subspaces $V=X_{-1}\supseteq X_0\supseteq X_1\supseteq\ldots$ as follows. If $X_{n-1}$ has been defined and there is some subspace $Y\subseteq X_{n-1}$ such that I has a strategy in $B^T_Y(\vec x_n)$ to play in $\A$, let $X_n=Y$, otherwise let $X_n=X_{n-1}$. Finally, let $X\subseteq V$ be any subspace such that $X\subseteq^*X_n$ for all $n$.

By construction and using Observation \ref{hered}, we then have for any $\vec x_n\in T$,  
\[\begin{split}
\e Y\subseteq X, &\text{ I has a strategy in $B^T_Y(\vec x_n)$ to play in $\A$ }\\
&
\equi\;\;
\text{I has a strategy in $B^T_{X_n}(\vec x_n)$ to play in $\A$}
\\
&
\equi\;\;
\a Y\subseteq X, \text{ I has a strategy in $B^T_{Y}(\vec x_n)$ to play in $\A$},
\end{split}\]
which proves the lemma.
\end{proof}


\begin{lemme}\label{davis0}
Suppose $\A\subseteq E^\infty$, $V\subseteq E$, $\vec v\in E^{<\infty}$ and a $(V,\vec v)$-rule $T\subseteq E^{<\infty}$ are given such that 
$$
\a W\subseteq V,  \text{ I has no strategy in $B^T_W(\vec v)$ to play in $\A$}.
$$
Then, if $|\vec v|$ is odd, there is $X\subseteq V$ such that for all $Y\subseteq X$ there is $x\in Y\cap T_{\vec v}$ satisfying 
$$
\a W\subseteq X, \text{ I has no strategy in  $B^T_W(\vec v\con x)$ to play in $\A$}.
$$
Similarly, if $|\vec v|$ is even, there is $X\subseteq V\cap T_{\vec v}$ such that, for any $x\in X$, 
$$
\a W\subseteq X,  \text{ I has no strategy in $B^T_W(\vec v\con x)$ to play in $\A$}.
$$
\end{lemme}

\begin{proof}
Using Lemma \ref{diag}, we begin by choosing $Z\subseteq V$ such that,  for all $\vec x\in T$,
\[\begin{split}
\e Y\subseteq Z, &\text{ I has a strategy in $B^T_Y(\vec x)$ to play in $\A$ }\\
&
\equi\;\;
\a Y\subseteq Z, \text{ I has a strategy in $B^T_Y(\vec x)$ to play in $\A$}.
\end{split}\]

Suppose first that $|\vec v|$ is odd. Since, for all $Y\subseteq Z$,  I has no strategy in $B^{T}_{Y}(\vec v)$ to play in $\A$, we see that for all $Y\subseteq Z$ player II must be able to play pick some $x\in  Y\cap T_{\vec v}$ for which  I still has no strategy in $B^{T}_{Y}(\vec v\con x)$ to play in $\A$, whereby also
$$
\a W\subseteq Z, \text{ I has no strategy in $B^{T}_{W}(\vec v\con x)$ to play in $\A$}.
$$
Letting $X=Z$ the result follows.

Suppose instead that $|\vec v|$ is even. Again I has no strategy in $B^{T}_{Z}(\vec v)$ to play in $\A$. Therefore, II must be able to play some $X\subseteq Z$ so that, for all $x\in X\cap T_{\vec v}$, player I has no strategy in $B^{T}_{Z}(\vec v\con x)$ to play in $\A$ and thus
$$
\a W\subseteq X, \text{ I has no strategy in $B^{T}_{W}(\vec v\con x)$ to play in $\A$}.
$$
Finally, since $T$ is also an $(X,\vec v)$-rule, there is some $n$ such that $X[n]\subseteq T_{\vec v}$, so by replacing $X$ with $X[n]$, we may assume that $X\subseteq V\cap T_{\vec v}$.
\end{proof}


\begin{lemme}\label{davis1}
Suppose $\A\subseteq E^\infty$, $V\subseteq E$, $\vec v\in E^{<\infty}$ and a $(V,\vec v)$-rule $T$ are given. Assume that 
$$
\a W\subseteq V, \text{    I has no strategy in $B^T_W(\vec v)$ to play in $\A$.}
$$
Then there is some $X\subseteq V$ and an $(X,\vec v)$-rule $S\subseteq T$ such that $[S]\cap {\rm Int}(\A)=\tom$ and
$$
\a W\subseteq X, \text{ I has no strategy in $B_Z^S(\vec v)$ to play in $\A$.}
$$
\end{lemme}

\begin{proof}
Using Lemma \ref{diag}, we begin by choosing $V_0\subseteq V$ so that,  for all $\vec x\in T$,
\begin{equation}\begin{split}\label{y}
\e Y\subseteq V_0, &\text{ I has a strategy in $B^T_Y(\vec x)$ to play in $\A$ }\\
&
\equi\;\;
\a Y\subseteq V_0, \text{ I has a strategy in $B^T_Y(\vec x)$ to play in $\A$}.
\end{split}\end{equation}

Now, enumerate the elements of $T$ of even length as $\vec x_0,\vec x_1,\ldots$ and define a sequence of subspaces $V_0=X_{-1}\supseteq X_0\supseteq X_1\supseteq\ldots$ as follows. If $X_{n-1}$ has been defined and there is some subspace $Y\subseteq X_{n-1}$ such that,  for all $y\in Y\cap T_{\vec x_n}$, I has no strategy in $B^T_{V_0}(\vec x_n \con y)$ to play in $\A$, then let $X_n=Y$ and otherwise $X_n=X_{n-1}$. Finally, pick some $V_1\subseteq V_0$ such that $V_1\subseteq^*X_n$ for all $n$. 

We conclude that, for all $\vec x\in T$ of even length,  
\begin{equation}\begin{split}\label{yy}
&    \e Y\subseteq V_1\;\; \a y\in Y\cap T_{\vec x}, \text{ I has no strategy in $B^T_{V_0}(\vec x \con y)$ to play in $\A$, }\\
& \saa\e n\;\; \a y\in V_1[n]\cap T_{ \vec x}, \text{ I has no strategy in $B^T_{V_0}(\vec x\con y)$ to play in $\A$. }
\end{split}\end{equation}

By a similar diagonalisation, we find some $X\subseteq V_1$, such that,  for all $\vec x\in T$ of odd length,
\begin{equation}\begin{split}\label{yyy}
&    \e Y\subseteq V_1\;\; \a y\in Y\cap T_{\vec x}, \text{ I has a strategy in $B^T_{{V_0}}(\vec x \con y)$ to play in $\A$, }\\
& \saa\e n\;\; \a y\in V_1[n]\cap T_{ \vec x}, \text{ I has a strategy in $B^T_{{V_0}}(\vec x\con y)$ to play in $\A$. }
\end{split}\end{equation}

We define
 $$
 S=\{\vec x\in T\del \text{I has no strategy in $B^T_X(\vec x)$ to play in } \A\}
 $$
 and note that $\vec v \in S$. We claim that $S$ is an $(X,\vec v)$-rule.

Suppose first $\vec x\in S$  has  even length. Then I has no strategy in $B^T_X(\vec x)$ to play in $\A$ and so II must be able to play some $Y\subseteq X$ such that no matter which $y\in Y\cap T_{\vec x}$ I plays, I still has no  strategy in $B^T_X(\vec x\con y)$ to play in $\A$. Therefore, by (\ref{y}) and (\ref{yy}), there is $n$ such that, for all $y\in X[n]\cap T_{\vec x}$,  I also has no strategy in $B^T_X(\vec x\con y)$ to play in $\A$. As also $T$ is an $(X,\vec  x)$-rule, choosing  $n$ large enough, we can ensure that $X[n]\subseteq T_{\vec x}$, whereby also $X[n]\subseteq  S_{\vec x}$.

Now, assume instead $\vec x\in S$  has  odd length. As I has no strategy in $B^T_X(\vec x)$ to play in $\A$, we see that, for any $n$, there is some $y\in X[n]\cap T_{\vec x}$ such that  I still has no strategy in $B^T_X(\vec  x\con y)$ to play in $\A$. 
Thus, by (\ref{y}) and the contrapositive of (\ref{yyy}), for any $Y\subseteq X$, there is some $y\in Y\cap T_{\vec x}$ such that I has no strategy in $B^T_X(\vec x\con y)$ to play in $\A$, i.e., for any $Y\subseteq X$ there is $y\in Y$ with ${\vec x\con y}\in S$. 
Thus, $S$ is an $(X,\vec v)$-rule.

As ${\rm Int}(\A)$ is open, by the definition of $S$, we have  $ [S]\cap {\rm Int}(\A)=\tom$. Also, if $W\subseteq X$ and $\sigma$ were strategy for player I in $B_W^S(\vec v)$ to play in $\A$, then we could obtain a strategy for I in $B^T_W(\vec v)$ to play in $\A$ as follows: I uses the strategy $\sigma$ until, if ever, he encounters the first position $\vec x\notin S_{\vec v}$.  At this point, by the definition of $S$,  he can shift to a strategy in $B^T_W(\vec v\con \vec x)$ for playing in $\A$. On the other hand, if the shift never happens, the outcome will lie in $\A$ anyway. Since I has no strategy in $B^T_W(\vec v)$ to play in $\A$, he cannot have one in $B^S_W(\vec v)$ either.
\end{proof}


For $D\subseteq E^{<\infty}$, we let  $\OO(D)$ denote the open set $\OO(D)=\bigcup_{\vec x\in D}N_{\vec x}\subseteq E^\infty$. We shall use the following elementary observation in the proof of Claim \ref{xxxxx} below. 
\begin{lemme}\label{extensions}
Suppose that $C,R\subseteq E^{<\infty}$ and that $C$ is closed under end-extensions, i.e., that if $\vec x\sqsubseteq \vec y$ and $\vec x\in C$, then also $\vec y\in C$. Assume that $\vec w\in R$ and that ${\bf x}\in N_{\vec w}\cap \OO(C\cap R)$. Then there is $\vec x\in C\cap R$ such that $\vec w\sqsubseteq \vec x\sqsubset {\bf x}$. 
\end{lemme}
\begin{proof}To see this, pick some $\vec y\in C\cap R$ with $\vec y\sqsubset {\bf x}$. If $\vec w\sqsubseteq \vec y$, we are done. If instead $\vec y\sqsubset\vec w$, then, as $C$ is closed under end-extension and $\vec w\in R$, we have $\vec w\in C_n\cap R$ and can let $\vec x=\vec w$. 
\end{proof}

\begin{lemme}\label{davis2}
Suppose $\A\subseteq E^\infty$, $\G\subseteq \A$ is a $G_\delta$ subset, $V\subseteq E$, $\vec v\in E^{<\infty}$ and a $(V,\vec v)$-rule $T$ are given. 
Assume that 
$$
\a W\subseteq V, \text{ I has no strategy in $B^T_X(\vec v)$ to play in $\A$.}
$$
Then there is some $X\subseteq V$ and an $(X,\vec v)$-rule $S\subseteq T$  such that $ [S]\cap \G=\tom$ and
$$
\a W\subseteq X,  \text{  I has no strategy in $B_Z^S(\vec v)$ to play in $\A$.}
$$
\end{lemme}

\begin{proof}Since $\G$ is $G_\delta$, we can find $E^{<\infty}\supseteq C_1\supseteq C_2\supseteq\ldots$ such that $\G=\bigcap_n\OO(C_n)$ and where each $C_n$ is closed under end-extensions.

We say that $\vec x$ {\em accepts} $X$ if there is an $(X,\vec x)$-rule $S\subseteq T$ such that $[S]\cap \G=\tom$ and, for all $W\subseteq X$,  I has no strategy in $B^S_W(\vec x)$ to play in $\A$.

Notice that if $Y\subseteq^* X$ and $\vec x$ accepts $X$, as witnessed by an $(X,\vec x)$-rule $S\subseteq T$, then $S$ also witnesses that $\vec x$  accepts $Y$. Therefore, by a simple diagonalisation as in the proof of Lemma \ref{diag}, we can find some $X_0\subseteq V$ such that, for all $\vec x\in T$, 
\begin{equation}\label{x}\begin{split}
\e Y\subseteq X_0, \;\vec x \text{ accepts }Y\quad
\equi\quad
\a Y\subseteq X_0, \;\vec x \text{ accepts }Y.
\end{split}\end{equation}
Let also 
$$
R=\{\vec x\in E^{<\infty}\del  \vec x \text{ does not accept }X_0\}.
$$

Again, by applying Lemma \ref{diag} successively for every $n$,  we find some $X_1\subseteq X_0$ such that,  for all $\vec x\in T$ and $n$,
\begin{equation}\label{xx}\begin{split}
\e Y\subseteq X_1,& \text{  I has a strategy in $B^T_Y(\vec x)$ to play in $\A\cup \OO(C_{n}\cap R)$}\\
\equi\;\;
&\a Y\subseteq X_1, \text{  I has a strategy in $B^T_Y(\vec x)$ to play in $\A\cup \OO(C_{n}\cap R)$}.
\end{split}\end{equation}

We say that $(\vec x,n)$ {\em likes} $X\subseteq X_1$ if there is an $(X,\vec x)$-rule $S\subseteq T$ such that $[S]\cap \OO(C_n\cap R)=\tom$ and, for all $W\subseteq X$, I has no strategy in $B^S_W(\vec x)$ to play in $\A$.

Again, we see that if $(\vec x,n)$ likes $X$ and $Y\subseteq^*X$, then $(\vec x,n)$ also likes $Y$. So, by yet another diagonalisation, we find some $X\subseteq X_1$ such that, for all $\vec x\in T$ and $n$,
\begin{equation}\label{xxx}\begin{split}
\e Y\subseteq X, \;(\vec x, n) \text{ likes }Y\quad
\equi\quad
\a Y\subseteq X, \;(\vec x, n) \text{ likes }Y.
\end{split}\end{equation}

\begin{claim}\label{xxxx}
Suppose  $\vec x\in R\cap T$. Then, for any $n$,  I has a strategy in $B^T_X(\vec x)$ to play in $\A\cup \OO(C_{n}\cap R)$.
\end{claim}

\begin{proof}Assume that $\vec x\in T$ and that, for some $n$, I has no strategy in $B^T_X(\vec x)$ to play in $\A\cup \OO(C_{n}\cap R)$. Then, by (\ref{xx}), for all $Y\subseteq X$, I has also no strategy in $B^T_Y(\vec x)$ to play in $\A\cup\OO(C_n\cap R)$. 

As $ \OO(C_{n}\cap R)\subseteq{\rm Int}(\A\cup \OO(C_{n}\cap R))$,  by Lemma \ref{davis1}, $(\vec x,n)$ likes some $Y\subseteq X$ and thus by (\ref{xxx}) also likes $X$. In other words, there is an $(X,\vec x)$-rule $S\subseteq T$
such that $[S]\cap\OO(C_{n}\cap R)=\tom$ and, for all $W\subseteq X$,  I has no strategy in $B^S_W(\vec x)$ to play in $\A$.

Suppose that, in a run of the game $B^S_X(\vec x)$,  all positions $\vec y$ satisfy $\vec x\con\vec y\notin C_{n}$. Then the outcome of that run will not lie in $\OO(C_n)$ and hence not in $\G$ either. 

On the other hand, if ever a position $\vec y$ is reached such that $\vec x\con \vec y\in C_{n}$, then, as $[S]\cap\OO(C_{n}\cap R)=\tom$, actually $\vec x\con \vec y\in C_{n}\setminus R$, which means that $\vec x\con \vec y$ accepts $X_0$ and hence by (\ref{x})  also accepts $X$. Therefore, there is an $(X,\vec x\con \vec y)$-rule $S'\subseteq T$ such that $[S']\cap \G=\tom$ and, for all $W\subseteq X$,  I has no strategy in $B^{S'}_W(\vec x\con \vec y)$ to play in $\A$. 

Thus, we can modify $S$ such that if ever such a $\vec y$ is encountered, the game continues inside the $(X,\vec x\con \vec y)$-rule $S'$ (depending on $\vec y$). The resulting $(X,\vec x)$-rule $\tilde S\subseteq T$
then satisfies $[\tilde S]\cap \G=\tom$ and, for all $W\subseteq X$,  I has no strategy in $B^{\tilde S}_W(\vec x)$ to play in $\A$. But this shows that $\vec x$ accepts $Y$ and so $\vec x\notin R$, proving our claim.
\end{proof}

\begin{claim}\label{xxxxx}
If $\vec v\in R\cap T$, then I has a strategy in $B^T_X(\vec v)$ to play in $\A$.
\end{claim}

\begin{proof}
Assume $\vec v\in R\cap T$. We describe a strategy for I in $B^T_X(\vec v)$ to play in $\A$. 

First, by Claim \ref{xxxx}, I has a strategy in $B^T_X(\vec v)$ to play in $\A\cup \OO(C_{1}\cap R)$, and so I follows this strategy until, if ever, a first position $\vec x_1$ is encountered, such that $\vec v\con \vec x_1\in C_1\cap R$.  If no such position is encountered, then, by Lemma \ref{extensions}, then outcome will not lie in $\OO(C_1\cap R)$ but must lie in $\A$. So suppose instead that $\vec x_1$ exists and let $n_1$ be the supremum of $n$ such that $\vec v\con \vec x_1\in C_n\cap R$. 
If $n_1=\infty$, then $N_{\vec v\con \vec x_1}\subseteq \bigcap_n\OO(C_n)=\G\subseteq \A$ and thus I can finish the game randomly. 
If, on the other hand,  $n_1$ is finite, then by Claim \ref{xxxx} player  I has a  strategy in $B^T_X(\vec v\con\vec x_1)$ to play in $\A\cup \OO(C_{n_1+1}\cap R)$ and he will continue with this strategy until, if ever, a further position $\vec x_2$ is encountered such that $\vec v\con \vec x_1\con \vec x_2\in C_{n_1+1}\cap R$. 
Again, if this does not happen, then outcome of the game will lie in $\A$. So assume instead that $\vec x_2$ exists and let $n_2$ be the supremum of $n$ such that $\vec v\con \vec x_1\con \vec x_2\in C_n\cap R$. 
If $n_2=\infty$, then,  as before, $N_{\vec v\con \vec x_1\con \vec x_2}\subseteq \A$ and I can finish the game randomly.
If $n_2$ is finite, I has a  strategy in $B^T_X(\vec v\con\vec x_1\con \vec x_2)$ to play in $\A\cup \OO(C_{n_2+1}\cap R)$, etc.

 Now, if this procedure is repeated infinitely often, we have that the outcome $\vec v\con \vec x_1\con\vec x_2\con\ldots$ belongs to $\bigcap_n\OO(C_n)=\G\subseteq \A$. On the other hand, if the procedure is only repeated finitely often, we instead have that the outcome lies in $\A$.
\end{proof}
Since the conclusion of Claim \ref{xxxxx} contradicts the assumption of the lemma, it follows that $\vec v\notin R\cap T$ and thus $\vec v\notin R$, implying the conclusion of the lemma.
\end{proof}


\begin{thm}\label{gdeltasigma}
Suppose $\F\subseteq E^\infty$ is $F_{\sigma\delta}$ or $G_{\delta\sigma}$. Then there is $X\subseteq E$
such that either
\begin{enumerate}
  \item II has a strategy in $A_X$ to play in $\sim\F$, or
  \item I has a strategy in $B_X$ to play in $\F$.
\end{enumerate}
\end{thm}

\begin{proof}By symmetry, we can suppose that $\F$ is $G_{\delta\sigma}$.
Assume that for no $W\subseteq E$ does I have a strategy in $B_W$ to play in $\F$ and let $\G_0\subseteq \G_1\subseteq \G_2\subseteq \ldots$ be $G_\delta$ sets with union $\F$. Fix also an enumeration $\vec v_0,\vec v_1,\vec v_2,\ldots$ of $E^{<\infty}$ such that $n\leqslant m$ whenever $\vec v_n\sqsubseteq \vec v_m$.

We will construct a subspace $W\subseteq E$, a $(W,\tom)$-rule $T$ and, for every $\vec v_k\in T$, a $(W,\vec v_k)$-rule $S_k$ such that the following conditions are satisfied, 
\begin{itemize}
\item[(i)] if $\vec v_k\sqsubseteq \vec v_l$ with $\vec v_k,\vec v_l\in T$, then $S_l\subseteq S_k$, 
\item[(ii)] if $\vec v_k\sqsubset \vec v_l$ with $\vec v_k,\vec v_l\in T$, then $[S_l]\cap \G_k=\tom$.
\end{itemize}
We claim that then $T$ is a quasistrategy for II in $A_W$ to play in $\sim \F$. To see this, since $T$ is a $(W,\tom)$-rule and hence a quasistrategy for II in $A_W$, we only need to show that $[T]\cap \G_l=\tom$ for every $l$. So fix $l$ and assume that $\vec v_{k_0}\sqsubset \vec v_{k_1}\sqsubset \ldots$ for some $\vec v_{k_0}, \vec v_{k_1},\ldots \in T$. Pick some $k_n\geqslant l$ and note that for any $m>n$,
$$
\vec v_{k_m}\in S_{k_m}\subseteq S_{k_{n+1}},
$$
whence ${\bf v}=\bigcup_m\vec v_{k_m}\in [S_{k_{n+1}}]$. Since $ [S_{k_{n+1}}]\cap \G_l=\tom$, the claim follows.

It thus remains to construct $W$, $T$ and $S_k$ as above. This will be done in stages along with the construction of an auxiliary sequence $E\supseteq X_0\supseteq X_1\supseteq \ldots$ of subspaces.

At the outset of stage $k\geqslant 0$, we assume that either $\vec v_k\notin T$ or that $\vec v_k\in T$ and $S_{k}$ is an $(X_k,\vec v_k)$-rule such that 
$$
\a W\subseteq X_k, \text{ I has no strategy in $B^{S_{k}}_W(\vec v_k)$ to play in $\F$}.
$$
Supposing that $\vec v_k\in T$, we will then choose $X_{k+1}\subseteq X_k$, decide which $\vec v_l=\vec v_k\con x$ belong to $T$ (by necessity $l> k$) and, for those $\vec v_l$, define $S_l$ in such a way that 
\begin{itemize} 
\item[(1)] for every $\vec v_l=\vec v_k\con x\in T$, $S_l$ is an $(X_{k+1},\vec v_l)$-rule with $[S_l]\cap \G_k=\tom$ and $S_l\subseteq S_k$,
\item[(2)] for all $\vec v_l=\vec v_k\con x\in T$ and $W\subseteq X_{k+1}$, I has no strategy in $B^{S_l}_W(\vec v_l)$ to play in $\F$.
\end{itemize}
Moreover, if $|\vec v_k|$ is odd, we will ensure that
\begin{itemize}
\item[(3)] for every $V\subseteq  X_{k+1}$, there is $x\in V$ such that $\vec v_k\con x\in T$,  
\end{itemize}
and, if $|\vec v_k|$ is even, that 
\begin{itemize}
\item[(3')] $\vec v_k\con x\in T$ for all $x\in X_{k+1}$.
\end{itemize}
Since also the sequence $X_0\supseteq X_1\supseteq \ldots$ is decreasing, it follows that, if $\vec v_l=\vec v_k\con x\in T$, then at stage $l$,   $S_l$ is an $(X_{l},\vec v_l)$-rule and, for every $W\subseteq X_{l}$, I has no strategy in $B^{S_l}_W(\vec v_l)$ to play in $\F$ and thus the assumptions of stage $l$ are verified.

\

\noindent{ \bf Initial stage:} We begin by putting $\vec v_0=\tom \in T$ and set $S_{0}=E^{<\infty}$, $X_0=E$. Then, by our initial assumptions, $S_0$ is an $(X_0,\vec v_0)$-rule such that, for all $W\subseteq X_0$, I has no strategy in $B_W=B_W^{S_0}(\vec v_0)$ to play in $\F$.

\noindent{ \bf Stage $k\geqslant 0$:} If $\vec v_k\notin T$, set $X_{k+1}=X_k$ and proceed to stage $k+1$. Suppose on the other hand that $\vec v_k\in T$ and that $S_{k}$ is an $(X_k,\vec v_k)$-rule such that 
$$
\a W\subseteq X_k, \text{ I has no strategy in $B^{S_{k}}_W(\vec v_k)$ to play in $\F$}.
$$
Now, using Lemma \ref{davis2}, we find $Y_{k}\subseteq X_k$ and a $(Y_k,\vec v_k)$-rule $R\subseteq S_{k}$ such that $[R]\cap \G_k=\tom$ and 
$$
\a W\subseteq Y_k, \text{ I has no strategy in $B^{R}_W(\vec v_k)$ to play in $\F$}.
$$

\noindent{ \bf Case 1, $|\vec v_k|$ is odd:} 
By Lemma \ref{davis0} we can find  a further $Z_k\subseteq Y_k$  such that, for all $V\subseteq Z_k$, there is $x\in V\cap R_{\vec v}$ satisfying 
$$
\a W\subseteq Z_k, \text{ I has no strategy in $B^{R}_W(\vec v_k\con x)$ to play in $\F$}.
$$
So put $\vec v_k\con x\in T$ if and only if both $x\in Z_k\cap R_{\vec v}$ and 
$$
\a W\subseteq Z_k, \text{ I has no strategy in $B^{R}_W(\vec v_k\con x)$ to play in $\F$}.
$$
Also,  for every $\vec v_l=\vec v_k\con x\in T$, let $S_l=R$. Finally, let $X_{k+1}=Z_k$. Then conditions (1), (2) and (3) are verified.

\noindent{ \bf Case 2, $|\vec v_k|$ is even:} 
By Lemma \ref{davis0} we can find  a further $X_{k+1}\subseteq Y_k\cap R_{\vec v}$ such that, for any $x\in X_{k+1}$, 
$$
\a W\subseteq X_{k+1},  \text{ I has no strategy in $B^{R}_W(\vec v\con x)$ to play in $\F$}.
$$
So put $\vec v_k\con x\in T$ if and only if $x\in X_{k+1}$. Also, for every $\vec v_l=\vec v_k\con x\in T$, let $S_l=R$. Then conditions (1), (2) and (3') are verified. 

\

At the end of the construction, we let $W\subseteq E$ be any space such that $W\subseteq^*X_k$ for all $k$. Now, $\tom \in T$. Also, by (3), if $\vec v\in T$ with  $|\vec v|$ even and $V\subseteq W$, then there is $x\in V$ such that $\vec v\con x\in T$. Similarly, by (3'), if $\vec v\in T$ with  $|\vec v|$ even, then there is an $n$ such that, for all $x\in W[n]$, $\vec v\con x\in T$. So $T$ is a $(W,\tom)$-rule. Finally, for $\vec v_l\in T$, $S_l$ is an $(X_{l},\vec v_l)$-rule and thus also a $(W,\vec v_l)$-rule and conditions (i) and (ii) are ensured by (1).
\end{proof}

\end{document}